\documentclass[11pt]{amsart}
\usepackage{amssymb,latexsym,amsmath,amsfonts}
\usepackage[mathscr]{eucal}
\usepackage{mathrsfs}
\usepackage{amscd}
\usepackage{amsthm}
\usepackage{amsxtra}
\usepackage[nointegrals]{wasysym}
\usepackage{bm}
\usepackage{calc}
\usepackage[usenames,dvipsnames]{xcolor}
\usepackage{hyperref}
\hypersetup{colorlinks=true, citecolor=MidnightBlue, linkcolor=BrickRed}
\newtheorem{theorem}{Theorem}

\newcommand{\beas}{\begin{eqnarray*}}
\newcommand{\eeas}{\end{eqnarray*}}
\newcommand{\bes} {\begin{equation*}}
\newcommand{\ees} {\end{equation*}}
\newcommand{\be} {\begin{equation}}
\newcommand{\ee} {\end{equation}}
\newcommand{\bea} {\begin{eqnarray}}
\newcommand{\eea} {\end{eqnarray}}

\newcommand{\zt}{\zeta}

\newcommand{\om}{\omega}

\newcommand\partl[2]{\dfrac{\partial{#1}}{\partial{#2}}}

\newcommand{\D}{\mathbb{D}}
\newcommand{\Dbar}{\overline{\mathbb{D}}}
\newcommand{\tor}{\mathbb{T}^2}

\newcommand{\cont}{\mathcal{C}}

\newcommand{\bdy}{\partial}
\newcommand{\rea}{\operatorname{Re}}

\newcommand{\wh}{\widehat}
\newcommand{\G}{\mathcal{G}}
\newcommand{\Z}{\mathcal{Z}}

\newcommand{\CC}{\mathbb{C}^2}
\newcommand{\Cn}{\mathbb{C}^n}
\newcommand{\C} {\mathbb{C}}

\newcommand{\std}{_{\operatorname{st}}}

\begin{document}
\title{	A nonpolynomially convex isotropic two-torus with no attached discs}
\author{Purvi Gupta}
\address{Department of Mathematics, University of Western Ontario, London, Ontario N6A 5B7, Canada}
\email{pgupta45@uwo.ca}

\begin{abstract}
We show --- with the means of an example in $\C^3$ --- that Gromov's theorem on the presence of attached holomorphic discs for compact Lagrangian manifolds is not true in the isotropic (subcritical) case, even in the absence of an obvious obstruction, i.e, polynomial convexity.	  
\end{abstract}
\maketitle
A compact set $X\subset\Cn$ is called {\em polynomially convex} if, for every $z\notin X$, there is a holomorphic polynomial $P$ such that $|P(z)|>\sup_{x\in X}|P(x)|$. It is known that no real compact $n$-dimensional submanifold $M\subset\Cn$ (without boundary) can be polynomially convex. In the particular case when the inclusion $\iota:M\hookrightarrow~\Cn$ is maximally {\em isotropic}	 (or Lagrangian) with respect to $\om\std=i\sum_{1}^{n}dz_j\wedge~ d\overline{z_j}$, i.e., $\iota^*(\om\std)=0$, Gromov (\cite{Gr85}) proves a stronger statement: there is a {\em holomorphic disc attached to} $M$, i.e., there is a 	nonconstant holomorphic map from the unit disc $\D$ to $\Cn$ that is continuous up to the boundary and maps $\bdy\D$ into $M$. Gromov's result is not true in the subcritical case (when $\dim M<n$) as there are several examples of polynomially convex isotropic surfaces in $\C^3$. 

In this note, we confirm that Gromov's result is not true in the subcritical case, even in the absence of polynomial convexity. We produce an explicit nonpolynomially convex two-torus $T\subset\C^3$ that is isotropic with respect to $\om\std$, but has no holomorphic discs attached to it. Note that the isotropicity of $T$ implies that it is both totally real and rationally convex (see \cite{DuSi95}). Examples of totally real tori with no attached holomorphic discs have been given by Alexander (\cite{Al99}) and Duval-Gayet (\cite{DuGa14}) in $\CC$, but such examples cannot be rationally convex in view of Duval-Sibony (see \cite[Theorem~3.1]{DuSi95}) and Gromov's result. In the case of manifolds with boundary, Duval has constructed an example of a nonpolynomially convex Lagrangian surface in $\CC$ that has no attached discs (see \cite{Du91} or \cite{DuSi95}).

\noindent {\em Note}. It was brought to the author's attention --- after the first version of this note was posted --- that a stronger example was constructed in the $\cont^\infty$-category in \cite{IzSaWo16}. The example therein has a nontrivial hull that does not contain {\em any} analytic structure. To the best of our knowledge, no such example is known in the real-analytic category. The example in this note is real-analytic, but contains a holomorphic annulus in its hull.
 
\begin{theorem}
There is a real-analytic two-torus in $\C^3$ that is isotropic with respect to $\om\std$, not polynomially convex, but has no holomorphic discs attached to it. 
\end{theorem}
\begin{proof} Let $p(z,w):=1-4z^2+4w^2-z^2w^2$ and 
	\bes	
		T:=\left\{\Big(z,w,\rea p(z,w)\Big)\in\C^3:z,w\in\bdy\D\right\}.
	\ees
Being the graph of a real-valued function on the torus $\tor:=\bdy\D\times\bdy\D$, $T$ is isotropic with respect to $\om\std$. We will show that $T$ is not polynomially convex, and its polynomial hull (defined below) consists of $T$ and an attached annulus. 

Before we proceed, we fix some notation. If $A\subset\Dbar^2$ and $f:\Dbar^2\rightarrow\C$, then $\G_f(A)=\{(z,w,f(z,w))\subset\C^3:(z,w)\in A\}$ denotes the graph of $f|_A$. If $\zt\in\Dbar^2$, $\G_f(\{\zt\})$ is simplified to $\G_f(\zt)$. For a compact $X\subset\Cn$, the {\em polynomial hull} of $X$ is the set $\wh X=\{z\in\Cn:|P(z)|\leq\sup_{x\in X}|P(x)|\}$. 

Now, let $f(z,w):=\rea(p(z,w))$. In our notation, $T=\G_f(\tor)$. We first consider $T_1:=\G_{\:\overline p}(\tor)$. We claim that
	\be\label{eq_hull}
		\wh{T_1}=T_1\cup\G_p(\Z), 
	\ee
where $\Z=\left\{(z,w)\in\Dbar^2:w^2=\frac{4z^2-1}{4-z^2}\right\}$. Since $p|_{\Z}\equiv 0$, $\G_p(\Z)$ is isomorphic to $\Z$. Moreover, by a computation due to Rudin (see \cite[proof of Theorem~B]{Ru69}) $\Z$ is a connected finite Riemann surface of genus $0$ with two boundary components in $\tor$ --- i.e, $\G_p(\Z)$ is an annulus attached to $T_1$. 

To prove \eqref{eq_hull}, we use a technique due to Jimbo (see \cite{Ji05}). Following the notation in \cite{Ji05}, let 
	\beas
		&& h(z,w)=(zw)^{-2}(z^2w^2-4w^2+4z^2-1);\\
		&& L=(\Dbar\times\{0\})\cup(\{0\}\times\Dbar); \text{and}\\
		&& V=\{(z,w)\in\Dbar^2\setminus(\tor\cup L):\overline{p(z,w)}=h(z,w)\}.
	\eeas
Note that $h(z,w)=\overline{p(z,w)}$ on $\tor$. Next, we compute
	\beas
		\Delta(z,w)=
		\left|
\begin{matrix}
    \partl{p}{z}(z,w) & \partl{p}{w}(z,w) \\
							& \\
     \partl{h}{z}(z,w) & \partl{h}{w}(z,w) \\
   \end{matrix}
\right|
=		\left|
\begin{matrix}
    -8z-2zw^2 &  8w-2z^2w \\
							& \\
     \dfrac{8}{z^3}+\dfrac{2}{z^3w^2} & -\dfrac{8}{w^3}+\dfrac{2}{z^2w^3} \\
   \end{matrix}
\right|,
	\eeas
to obtain $\Delta(z,w)=-16(zw)^{-3}(z-iw)(z+iw)p(z,w)$. Setting $q_1=(z-iw)$, $q_2=z+iw$, $q_3=p(z,w)$, and $Q_j:=\{(z,w)\in\tor:q_j(z,w)=0\}$, $1\leq j\leq 3$, we have that
\bea\label{eq_factors}
	Q_1&=&\{(z,iz)\subset\tor:z\in\bdy\D\};\notag\\
	Q_2&=&\{(z,-iz)\subset\tor:z\in\bdy\D\}; \text{and}\\
	Q_3&=&\Z\cap\tor=\bdy\Z.\notag
\eea
In \cite{Ji05}, Jimbo showed that if $\Delta(z,w)\not\equiv 0$ on $\D^2\setminus L$ and 
	\bes
		J:=\{1\leq j\leq 3:\emptyset\neq Q_j\neq \wh{Q_j},\: \wh{Q_j}\setminus(\tor\cup L)\subset V\}\neq \emptyset,
	\ees
then 
	\bes
		\wh{\G_{\overline{p}}(\tor)}=\G_{\overline{p}}(\tor)
				\cup \bigcup_{j\in J}\{(z,w,\overline{p(z,w)}):(z,w)\in\wh{Q_j}\},
	\ees
and $p$ restricts to a constant on each $\wh{Q_j}$, $j\in J$. In view of \eqref{eq_factors}, $J=\{3\}$, $\wh{Q_3}=\Z$ and, since $p|_{\Z}=\overline{p}|_{\Z}=0$, \eqref{eq_hull} holds --- i.e., there is only one annulus attached to $T_1$. Since $\tor$ is totally real and rationally convex, and $\overline{p}$ is smooth, $T_1=\G_{\overline{p}}(\tor)$ is totally real and rationally convex. Due to a result by Duval and Sibony (see \cite{DuSi95}), $T_1$ is isotropic with respect to some K{\"a}hler form on $\C^3$. But, $\iota^*(\om\std)\neq 0$, where $\iota:T_1\hookrightarrow\C^3$ is the inclusion map. 

We now return to $T:=\G_{f}(\tor)$. Note that the algebraic isomorphism
	\bes
		F(z,w,\eta)\mapsto \left(z,w,\frac{1}{2}(\eta +p(z,w))\right)
	\ees
maps $T_1$ onto $T$ and fixes the variety $\G_{p}(\Z)$. Thus, $\wh T=F(\wh {T_1})=T\cup\G_p(\Z)$. As there are no nontrivial holomorphic discs attached to an annulus, there are none attached to $T$. 
\end{proof}

\bibliography{isotor}
\bibliographystyle{plain}

\end{document}